\documentclass[11pt,a4paper,reqno]{amsart}
\usepackage{vmargin,color}
\usepackage[latin1]{inputenc}
\usepackage{amsmath,amsfonts,amsthm,epsfig,graphicx}
\usepackage{caption}

\def\P{\mathcal{P}}

\def\N{\mathbb N}

\def\R{\mathbb R}

\def\e{\varepsilon}

\def\vphi{\varphi}

\def\ds{\displaystyle}
\newcommand{\measurerestr}{%
  \,\raisebox{-.127ex}{\reflectbox{\rotatebox[origin=br]{-90}{$\lnot$}}}\,%
}

\DeclareMathAlphabet{\mathup}{OT1}{\familydefault}{m}{n}
\newcommand{\dx}[1]{\mathop{}\!\mathup{d} #1}

\newtheorem{theorem}{Theorem}
\newtheorem{remark}{Remark}[section]
\newtheorem{definition}{Definition}[section]

\newtheorem{proposition}[theorem]{Proposition}
\newtheorem{lemma}[theorem]{Lemma}

\numberwithin{equation}{section}
\numberwithin{figure}{section}

\author{J. A. Carrillo}
\author{M. G. Delgadino}
\author{G. A. Pavliotis}
\address{Department of Mathematics, Imperial College London, London SW7 2AZ}
\email{carrillo@imperial.ac.uk, m.delgadino@imperial.ac.uk, g.pavliotis@imperial.ac.uk}
\thanks{JAC and MGD were partially supported by the EPSRC through grant number EP/P031587/1. GAP was partially supported by the EPSRC through grant numbers EP/P031587/1, EP/L024926/1, and EP/L020564/1.}
\title{A proof of the mean-field limit for $\lambda$-convex potentials by $\Gamma$-Convergence}

\begin{document}
\maketitle
\begin{abstract}
In this work we give a proof of the mean-field limit for $\lambda$-convex potentials using a purely variational viewpoint. Our approach is based on the observation that all evolution equations that we study can be written as gradient flows of functionals at different levels: in the set of probability measures, in the set of symmetric probability measures on $N$ variables, and in the set of probability measures on probability measures. This basic fact allows us to rely on $\Gamma$-convergence tools for gradient flows to complete the proof by identifying the limits of the different terms in the Evolutionary Variational Inequalities (EVIs) associated to each gradient flow. The $\lambda$-convexity of the confining and interaction potentials is crucial for the unique identification of the limits and for deriving the EVIs at each description level of the interacting particle system.
\end{abstract}

\section{Introduction}
In this work we give an alternative proof of the mean field limit for interacting particle systems of the form
\begin{align}
dX^i_t = - \nabla V(X^i_t) \dx{t} -\frac{1}{ N} \sum\limits_{i \neq j}^N\nabla H(X^i_t -X^j_t) \dx{t} + \sqrt{2} \, dB^i_t \, ,
\label{eq:sdes}
\end{align}
where the stochastic processes $X^i_t$, $i\in\{1,\dots,N\}$ take values in a domain $\Omega\subset\R^d$ (that can be the entire $\R^d$), $B_t^i$, $i\in\{1,\dots,N\}$ denote standard one dimensional independent Brownian motions, the interaction potential $H:\R^ d\to\R$ is assumed to be bounded below, symmetric, with certain conditions at $\infty$ in case $\Omega$ is unbounded, and $\lambda$-convex, and the confinement potential  $V:\Omega\to \R$ is bounded below and $\lambda$-convex in $\Omega$.

Let us denote by $\mu^N$ the $N$-particle probability density, which is symmetric due to exchangeability of the particle system, and let us denote by $\mu^N_1$ any of its one particle marginals. The classical and well known mean-field limit result by Sznitman \cite{Szn} shows that interacting particle systems with globally Lipschitz and bounded interactions are determined by a nonlinear Fokker-Planck evolution equation for the limit of the first marginal $\mu^N_1$ as $N\to\infty$, usually referred as the McKean-Vlasov equation. In the particular case in which these interactions are derived from potentials as in \eqref{eq:sdes}, one can work with locally Lipschitz or singular interactions once the behavior of the potentials at infinity is under control, see \cite{Mal0,BGM,BCC,Bol,GQ,salem2018gradient} and the references therein for related results.

In fact, under the hypotheses on the confining and interaction potentials in the first paragraph, the gradient flow approach developed in \cite{Villani,ambrosio2008gradient} can be used to show that the Cauchy problem for the formal mean-field limit of \eqref{eq:sdes}, given by the nonlinear McKean-Vlasov equation
\begin{equation}\label{mve}
\partial_t \rho+\nabla\cdot((\nabla V +\nabla H *\rho) \rho)=\Delta \rho,
\end{equation}
for $x \in \Omega$ and with no-flux boundary conditions on $\partial \Omega$, is well-posed in $\mathcal{P}_2(\Omega)$, the set of probability measures with bounded second moment in $\Omega$. Therefore, it is expected that the mean-field limit should hold in this setting, that is different from the classical setting of Sznitman \cite{Szn}.

Our strategy is to derive evolutions of gradient flows at three different levels: the first one at the level of the formal mean-field limit McKean-Vlasov equation \eqref{mve} just mentioned, the second one at the level of the $N$-particle probability density $\mu^N$ in the set of symmetric probability measures in the product space $\mathcal{P}_{sym}(\Omega^N)$, and finally the third one at the level of probability measures on $\mathcal{P}(\Omega)$, denoted by $\mathcal{P}(\mathcal{P}(\Omega))$; naturally, the empirical measure associated to \eqref{eq:sdes} is an element of this space. In all these spaces, we assume the equivalent growth condition to second bounded moments as for \eqref{mve} but we avoid the subscript $2$ for notational simplicity. We show that we can naturally construct these evolutions based on gradient flows using the $\lambda-$convexity of the confining and interaction potentials and that we can relate them by taking the limit $N\to\infty$ in a suitable manner.
To be more precise, given $X\in \mathcal{P}(\mathcal{P}(\Omega))$, we define 
\begin{equation*}
X^N=\int_{\P(\Omega)}\rho^{\otimes N}\;dX(\rho)\in \P_{sym}(\Omega^N)\,,
\end{equation*}
by duality as
\begin{equation}\label{aux}
\langle\phi, X^N\rangle_{C_b(\Omega^N),\mathcal{P}_{sym}(\Omega^N) }=\int_{\mathcal{P}(R)}\left(\int_{\Omega^N}\phi(x)\,d\rho^{\otimes N}(x) \right)\,dX(\rho)\,,
\end{equation}
for any $\phi\in C_b(\Omega^N)$, where $\rho^{\otimes N}$ represents the tensor product
$$
d\rho^{\otimes N}(x)=d\rho(x_1)d\rho(x_2)...d\rho(x_N).
$$
Then we can define rigorously our notion of convergence relating the sequence of $N$-particle
probability densities $\mu^N$ to objects living in $\mathcal{P}(\mathcal{P}(\Omega))$.

\begin{definition}\label{def:convergence1}
Given a sequence $\{\mu^N\}_{N\in \N}$, such that $\mu^N\in \P_{sym}(\Omega^N)$ for every $N\in \N$, and $X\in \P(\P(\Omega))$, we say that $\mu^N\to X$, if 
\begin{equation*}
\lim_{N\to \infty} \frac{1}{N}d_2^2(\mu^N, X^N)=0\, ,
\end{equation*}
where $d_2(\cdot, \cdot)$ denotes the $2-$Wasserstein distance.
\end{definition}
This notion of convergence was studied in \cite{hauray2014kac} and it implies the convergence of the one-particle marginal distributions towards a limiting density. Our main result can be summarized as follows.

\begin{theorem}\label{main result}
Given $\lambda\in \R$. We assume that $V:\Omega\to \R$ is bounded below and $\lambda$-convex in $\Omega$, and that $H:\R^ d\to\R$ is bounded below, symmetric, $\lambda$-convex and satisfies the doubling condition,
\begin{equation}\label{eq:doublinghyp}
\exists\, C>0:\;H(x+y)\le C(1+H(x)+H(y))\quad\forall x,\,y\in \Omega.    
\end{equation}
Given $\{\mu^N_0\}_{N\in \N}$ and $X_0\in \mathcal{P}(\mathcal{P}(\Omega))$, such that $\mu^N_0\in \P_{sym}(\Omega^N)$, $\mu^N_0\to X_0$ in the sense of Definition~\ref{def:convergence1}, and
\begin{equation}\label{eq:bounded energy}
\sup_{N\in \N}\frac{1}{N}\mathcal{F}^N[\mu^N_0]=\sup_{N\in \N} \left\{\frac{1}{N}\int_{\Omega^N} W^N(x) d\mu^N_0(x)\;dx+ \frac{1}{N}\int_{\Omega^N} \mu^N_0 \log(\mu^N_0)\;dx\right\}<\infty,
\end{equation}
where
\begin{equation*}
W^N=\sum_{i=1}^N V(x_i)+\frac{1}{2N}\sum_{i \ne j} H(x_i-x_j).
\end{equation*}
We consider $\mu^N:[0,\infty)\to \P_{sym}(\Omega^N)$ the unique gradient flow of $\mathcal{F}^N$ with initial condition $\mu^N_0$ under the $d_2$ metric. Then, for any $T>0$ we have 
\begin{equation*}
\lim_{N\to\infty}\sup_{t\in[0,T]}\frac{1}{N}d_2^2(\mu^N(t),X^N(t))=0,
\end{equation*}
where 
\begin{equation*}
X^N(t)=\int_{\P(\Omega)} (S_t\rho)^{\otimes N}\;dX_0(\rho),
\end{equation*}
and $S_t:\P(\Omega)\to \P(\Omega)$ is the (nonlinear) semigroup generated by the associated McKean-Vlasov-Fokker-Planck equation \eqref{mve}.
In particular, under the hypothesis of initial propagation of chaos $\mu^N_0\to \delta_{\rho_0}$, then we have the propagation of chaos uniformly over $t \in [0,T]$, and the mean-field limit holds: 
\begin{equation*}
\lim_{N\to\infty}\sup_{t\in[0,T]}\frac{1}{N}d_2^2(\mu^N(t),(S^t\rho_0)^{\otimes N})=0,
\end{equation*}
for every $T>0$, and consequently $\mu^N (t)\to \delta_{S^t\rho_0}$ and
$$
\lim_{N\to\infty} d_2^2(\mu_1^N(t),S^t\rho_0)=0,
$$
for all $t>0$, where $\mu_1^N$ is the first marginal of $\mu^N$.
\end{theorem}

\begin{remark}\
\begin{itemize}
\item The hypothesis of bounded energy \eqref{eq:bounded energy} as $N\to\infty$ in Theorem \ref{main result} is weaker than the well-preparedness for the initial data
$$
\lim_{N\to\infty}\frac{\mathcal{F}^N[\mu_N^0]}{N}=\mathcal{F}^\infty[X_0].
$$

\item The doubling hypothesis for $H$ \eqref{eq:doublinghyp} is only used in Lemma~\ref{lem:uniquegradpp}, which shows the well-posedness of the McKean-Vlasov-Fokker-Planck equation.

\item Our assumptions on the confining and interaction potentials $V$, include double well potentials such as $(1-|x|^2)^2$. For example, our results apply to the Desai-Zwanzig model~\cite{Dawson1983}. Note that $\lambda$-convexity of the potentials imply that in terms of regularity both potentials are at least locally Lipschitz. 
\end{itemize}
\end{remark}

We now comment on the relation between this work and other works on mean field limits for interacting diffusions. In addition to the already cited works on gradient flows, this paper is motivated by~\cite{messer1982statistical} in which a variational approach was adopted for the study of the mean field limit of the free energy functional for classical point particles in a box; see also~\cite{Kiessling1993} and more recent work on evolutionary Gamma convergence~\cite{serfaty2011}. In particular, our goal is to provide a complete, self-contained proof of a propagation of chaos result that relies only on analytical and variational arguments, in contrast to, e.g. probabilistic/martingale techniques~\cite{oelschlager1984}. We also mention an alternative approach based on coupling arguments~\cite{DEGZ2018} that also leads to a short, self-contained proof of uniform in time propagation of chaos results, see also related results on uniform in time propagation of chaos in \cite{salem2018gradient} for systems of weakly interacting diffusions. It should be mentioned, however, that the class of drifts for which the results in~\cite{DEGZ2018,salem2018gradient} are applicable, is broader to the $\lambda-$convex potentials that are covered by the techniques that are used in the present paper.

The rest of the paper is organized as follows. In Section~\ref{sec:prelim} we introduce several notations and transport distances at the different levels of description of the $N$-particle system. In Setion~\ref{sec:exploiting convexity} we exploit the $\lambda$-convexity to show convexity of the corresponding free energy at the $N$-particle symmetric probability density level. In Section~\ref{sec:metricconvergence} we summarize the characterization of the notion of convergence in Definition~\ref{def:convergence1}, together with compactness properties of curves in $\P_{sym}(\Omega^N)$. Section~\ref{sec:gammaconvergence} is devoted to the proof of the $\Gamma$-convergence of the involved functionals as $N\to\infty$ to the corresponding free energy defined on $\mathcal{P}(\mathcal{P}(\Omega))$. Finally, in Section~\ref{sec:uniquemain} we utilize the gradient flow theory on $\P_{sym}(\Omega^N)$ to define the corresponding evolution semigroups characterized by their Evolutionary Variational Inequalities leading to the passing to the limit as $N\to\infty$ in the EVIs and our main result. The identification of the limit uses again crucially the classical gradient flow theory in $\P_2(\Omega)$ for the McKean-Vlasov-Fokker-Planck equation \eqref{mve}.


\section{Preliminaries}\label{sec:prelim}
\subsection{Notation and Preliminary results}
Let us start by setting up a similar framework to Rougerie~\cite[Chapter 1]{rougerie2015finetti}. Given $\Omega\subset \R^d$ and $N\in \N$, the set $\Omega^N\subset \R^{dN}$ is given by the product of $N$ copies of $\Omega$. We say that a probability measure $\mu^N\in \mathcal{P}(\Omega^N)$ is symmetric, denoted by $\mu\in \mathcal{P}_{sym}(\Omega^N)$, if $\sigma_{\#}\mu^N=\mu^N$ for any permutation $\sigma$ of the N variables. In the literature, this property is referred as exchangeability. The n-th marginal, denoted by $\mu^N_n\in \mathcal{P}_{sym}(\Omega^n)$, is characterized by duality: 
\begin{equation}\label{def:marginal}
    \langle\psi,\mu^N_n\rangle_{C_b(\Omega^n),\mathcal{P}_{sym}(\Omega^n) } =\int_{\Omega^{N}}\psi(x_1,...,x_n)\;d\mu^N(x)\qquad\mbox{for any $\psi\in C_b (\Omega^n)$}.
\end{equation}
We note that by symmetry the marginal is independent of the variables we evaluate $\psi$ on.

We consider $T^N:\Omega^N\to \mathcal{K}^N\subset\mathcal{P}(\Omega^N)$ the map given by
\begin{equation*}
T^N(x_1,...,x_N)=\frac{1}{N}\sum_{i=1}^N\delta_{x_i},
\end{equation*}
where $\mathcal{K}^N$ is the set of probability measures given by the average of $N$ Dirac measures and coincides with the image of $T^N$.  We define the empirical measure associated to $\mu^N\in \mathcal{P}_{sym}(\Omega^N)$ as the image measure through $T^N$, i.e.
\begin{equation*}
\hat{\mu}^N=T^N_\#\mu^N\in \mathcal{P}(\mathcal{P}(\Omega)).
\end{equation*}

Note that taking $X$ as $\hat{\mu}^N\in \mathcal{P}(\mathcal{P}(\Omega))$ in \eqref{aux}, then $(\hat{\mu}^N)^n\in \mathcal{P}_{sym}(\Omega^n)$ is given by
\begin{equation*}
(\hat{\mu}^N)^n=\int_{\P(\Omega)}\rho^{\otimes n}\;d\hat{\mu}^N(\rho) = \int_{\Omega^N}\left(\frac{1}{N}\sum_{i=1}^N\delta_{x_i}\right)^{\otimes n}\;d\mu^N(x).
\end{equation*}

Using the previous notation we have the following result.
\begin{lemma}[Diaconis-Freedman \cite{diaconis1980finite}]\label{lem:DiaconisFreedman}
Given $n<N$ we have the following estimate for the total variation norm
\begin{equation*}
\|\mu^N_n-(\hat{\mu}^N)^n\|_{TV}\le 2\frac{n(n-1)}{N}.
\end{equation*}
\end{lemma}

For completeness, we provide a simple proof of this result.

\begin{proof}
Using the definition of the map $(\hat{\mu}^N)^n$, we have
\begin{equation*}
(\hat{\mu}^N)^n=\int_{\Omega^N}\left(\frac{1}{N^N}\sum_{\gamma\in \Gamma_N} \delta_{z_\gamma}\right)^{\otimes n}\;d\mu(z),
\end{equation*}
where $\Gamma_N$ is the set of maps from $\{1,...,N\}$ onto itself. Whilst we can rewrite
\begin{equation*}
\mu^N_n=\int_{\Omega^N}\left(\frac{1}{N!}\sum_{\sigma\in \Sigma_N} \delta_{z_\sigma}\right)^{\otimes n}\;d\mu(z),
\end{equation*}
where $\Sigma_N$ is the set of permutations of $\{1,...,N\}$. Now, counting the number of maps leaving invariant $N-n$ variables up to symmetries, we can compute that
\begin{equation*}
(\hat{\mu}^N)^n= \frac{N!}{n!N^n}\mu^N_n+\nu_n,
\end{equation*}
where $\nu_n$ is a positive measure on $\P_{sym}(\Omega^n)$. Hence,
\begin{equation*}
\int_{\Omega^n}\nu_n=\left(1-\frac{N!}{n!N^n}\right),
\end{equation*}
which implies that
\begin{equation*}
\int_{\Omega^n}|(\hat{\mu}^N)^n-\mu^N_n|\le 2\left(1-\frac{N!}{n!N^n}\right).
\end{equation*}
The estimate follows by noticing that
\begin{equation*}
1-\frac{N!}{n!N^n}\le \frac{n(n-1)}{N}.
\end{equation*}
\end{proof}

It will be useful to be able to easily distinguish between two members of $\mathcal{P}(\mathcal{P}(\Omega))$, just by looking at the symmetric measures they induce, see Eqn.~ \eqref{aux}.
\begin{lemma}[\cite{lions2007mean}]\label{lem:moments}
Let $X$ and $Y\in \P(\P(\Omega))$, then $X=Y$ if and only if for every $n\in\N$
\begin{equation}\label{eq:equalmarginals}
\int_{\P(\Omega)}\rho^{\otimes n}\;dX(\rho)=\int_{\P(\Omega)}\rho^{\otimes n}\;dY(\rho).
\end{equation}
\end{lemma}

\begin{proof}
We prove this Lemma by duality with bounded continuous functions $C_b(\P(\Omega))$. We consider the algebra of functionals $\{M_{k,\vphi}\}_{k\in \N,\, \vphi\in C_b(\Omega^k)}\subset C_b(\P(\Omega))$, defined by
\begin{equation*}
M_{k,\vphi}(\rho)=\int_{\Omega^k}\vphi \;d\rho^{\otimes k}.
\end{equation*}
By \eqref{eq:equalmarginals} and Fubini's theorem, we have that for any monomial $M_{k,\vphi}$,
\begin{equation}\label{eq:equalmonomials}
\begin{array}{rcl}
\ds\int_{\Omega^k}\vphi\;d\left(\int_{\P(\Omega)}\rho^{\otimes k}\;dX(\rho)\right)&=&\ds\int_{\Omega^k}\vphi\;d\left(\int_{\P(\Omega)}\rho^{\otimes k}\;dY(\rho)\right),\\
\ds\int_{\P(\Omega)} \left(\int_{\Omega^k}\vphi\;d\rho^{\otimes k}\right)\;dX(\rho)&=&\ds\int_{\P(\Omega)} \left(\int_{\Omega^k}\vphi\;d\rho^{\otimes k}\right)\;dY(\rho),\\
\ds\langle M_{k,\vphi}, X\rangle_{C_b(\P(\Omega)),\P(\P(\Omega))}&=&\langle M_{k,\vphi},Y\rangle_{C_b(\P(\Omega)),\P(\P(\Omega))}.
\end{array}
\end{equation}
By the general version of the Stone-Weierstrass Theorem, we have that the algebra of monomial functions on $\P(\Omega)$ is dense $C_b(\P(\Omega))$. Therefore, by the density of the monomials and \eqref{eq:equalmonomials}, we have that $X=Y$.
\end{proof}

\subsection{The Wasserstein distance and narrow convergence}
In the sequel, we need to consider the 2-Wasserstein distance in the space of probability measures defined over probability measures. Therefore, it is appropriate to give the definition of the 2-Wasserstein distance and state its properties for general complete separable metric spaces. This framework can be found in \cite[Chap. 7]{Villani} and \cite[Chap. 2]{MR3050280}, where a more detailed exposition and proofs can be found. 

Let $(S,D)$ be a Polish space, i.e. a complete, separable metric space. We denote by $P(S)$ the space of probability measures defined on $S$. We start by recalling the notion of narrow convergence. Given a sequence $\{\mu_{n}\}_{n\in\N}\subset \mathcal{P}(S)$, it narrowly converges to $\mu_{\infty}$, denoted by
$$
\mu_{n}\rightharpoonup\mu_{\infty},
$$
if
$$
\lim_{n\to\infty} \int_{S}f(x)\;d\mu_n(x)=\int_{S}f(x)\;d\mu_{\infty}(x)\qquad\mbox{for any$f\in C_b(S)$.}
$$
We also recall a standard application of Prohorov's theorem:
\begin{theorem}
Given a sequence $\{\mu_n\}_{n\in\N}\subset \mathcal{P}(S)$, assume that 
$$
\sup_{n\in \N} \int_{S} d^2(x,x_0)\;d\mu_n(x)<\infty\qquad\mbox{for some $x_0\in S$}.
$$
Then
$$
\{\mu_n\}_{n\in\N}\; \mbox{is relatively compact.}
$$
\end{theorem}

Given $\mu,$ $\nu\in\mathcal{P}(S)$, we define the 2-Wasserstein distance between the two measures by
$$
W_2(\mu,\nu)=\left(\inf_{\pi\in\Pi(\mu,\nu)}\int_{S\times S} D^2(x,y)\;d\Pi(x,y)\right)^{1/2},
$$
where 
$$
\Pi(\mu,\nu)=\left\{\pi\in \mathcal{P}(S\times S)\;:\;\pi(A\times S)=\mu(A)\; \mbox{and}\; \pi( S\times A)=\nu(A)\;\mbox{for any Borel set $A$}\right\}.
$$
We define
$$
\mathcal{P}_2(S)=\left\{\mu\in\mathcal{P}(S)\;:\; \int_{S}D^2(x,x_0)\;d\mu(x)<\infty \right\},
$$
where $x_0\in S$ is an arbitrary point.

Now we are ready to state the fundamental properties of the 2-Wasserstein distance.

\begin{theorem}\label{thm:basicproperties}\cite[Theorem 2.7]{MR3050280}
If $(S,D)$ is a complete, separable metric space, then the pair $(\mathcal{P}_2(S), W_2)$ is a complete, separable metric space. 
Moreover, given a sequence $\{\mu_n\}_{n\in\N}\subset \mathcal{P}_2(S)$, then
$$
W_2(\mu_n,\mu_{\infty})\to0 
$$
if and only if
$$
\mu_n\rightharpoonup\mu_{\infty} \qquad\mbox{and }\qquad \int_{S}D^2(x,x_0)\;d\mu_n(x)\to\int_{S}D^2(x,x_0)\;d\mu_\infty(x)\;\mbox{for some $x_0\in S$}.
$$ 
\end{theorem}

When $S=\Omega\subset\R^d$ with the usual Euclidean distance, we denote the 2-Wasserstein distance on $\mathcal{P}_2(\Omega)$ by $d_2$ to avoid confusion. Theorem~\ref{thm:basicproperties} shows that $(\mathcal{P}_2(\Omega),d_2)$ is a complete separable metric space. We also consider the 2-Wasserstein distance on the probability measures defined on $\mathcal{P}_2(\Omega)$, which we denote by $\mathfrak{D}_2$. Again, applying Theorem~\ref{thm:basicproperties} we obtain that $(\mathcal{P}_2(\mathcal{P}_2(\Omega)),\mathfrak{D}_2)$ is a complete separable metric space. To simplify the notation, in the rest of the paper we will omit the subscript $2$ in the definitions of the complete metric spaces and refer to them as $\mathcal{P}(\Omega)$ and $\mathcal{P}(\mathcal{P}(\Omega))$, respectively.

\section{Exploiting Convexity}\label{sec:exploiting convexity}
In this section we consider the family of free energies
\begin{equation*}
\mathcal{F}^N[\nu^N]=\int_{\Omega^N} W^N \;d\nu^N+\int_{\Omega^N}\nu^N\log(\nu^N)\;dx,
\end{equation*}
where
\begin{equation*}
W^N= \sum_{i=1}^N V(x_i)+\frac{1}{2N}\sum_{i\ne j} H(x_i-x_j).
\end{equation*}
We assume that $V$ is $\lambda$-convex on $\Omega$, while $H$ is symmetric and $\lambda$-convex. Our goal in this section is to show the following result.

\begin{lemma}\label{lem: convexity}
Under the hypothesis of Theorem~\ref{main result}, the potential
$$
W^N(x)= \sum_{i=1}^N V(x_i)+\frac{1}{2N}\sum_{i\ne j} H(x_i-x_j)
$$
is $\min(3\lambda,0)$-convex. Therefore, the functional
$$
\mathcal{F}^N:\mathcal{P}(\Omega^N)\to \R\cup \{\infty\}
$$
is $\min(3\lambda,0)$-convex on geodesics and generalized geodesics of the 2-Wasserstein distance.
\end{lemma}

Let us first make use of the structure of $W^N$ to observe that its Hessian satisfies the following identity.

\begin{lemma}
Given $V:\Omega\to \R$ and $H:\R^d\to R$, we consider $W^N:\Omega^{l}\to \R$ defined by 
$$
W^N(x_1,x_2,...,x_N)=\sum_{i=1}^N V(x_i)+\frac{1}{2N}\sum_{i\ne j} H(x_i-x_j)\qquad\mbox{with $x_i\in \Omega$.}
$$ 
Given any vector $v\in \Omega^{N}$, we denote its $(i-1)l+1$ to $il$ entries by $v_i\in R^l$. Then
$$
D^2 W^N [v,v]=\sum_{i=1}^N D^2V(x_i)[v_i,v_i]+\frac{1}{2N}\sum_{i\ne j} D^2H(x_i-x_j)[v_i-v_j,v_i-v_j].
$$
In particular, if there exists $\lambda\le 0$ such that $V$ and $H$ are $\lambda$-convex, then $W^N$ is $3\lambda$-convex.
\end{lemma}
\begin{proof}
Using the fact that differentiation commutes with summation, we only need to consider the second variation of each individual term. We notice that
$$
D^2_{Nl}V(x_i)[v,v]=D^2_{l}V(x_i)[v_i,v_i],
$$
where $D^2_{Nl}V(x_i)$ is the Hessian of $V(x_i)$ considered as a function from $\Omega^{N}$ to $\R$ while $D^2_{l}V(x_i)$  is the Hessian of $V$ considered as a function from $\Omega$ to $\R$, evaluated at $x_i$. Similarly,
$$
D^2_{Nl}H(x_i-x_j)[v,v]=D^2_lH(x_i-x_j)[v_i-v_j,v_i-v_j].
$$
The formula for the Hessian of $W^N$ follows by summing up these identities.

We know show convexity. First, if $\lambda\ge 0$, convexity follows. Assume now that $\lambda<0$; we notice that by applying the formula and using the $\lambda$-convexity of $V$ and $H$, we obtain
\begin{align*}
D^2 W^N[v,v]&\ge \sum_{i=1}^N D^2V(x_i)[v_i,v_i]+\frac{1}{2N}\sum_{i\ne j} D^2H(x_i-x_j)[v_i-v_j,v_i-v_j]\\
&\ge \sum_{i=1}^N \lambda |v_i|^2+\frac{1}{2N}\sum_{i\ne j} \lambda|v_i-v_j|^2 \ge \lambda \left(1+ \frac{1}{N} \sum_{i\ne j} |v_i|^2+|v_j|^2\right).
\end{align*}
Taking the infimum in the previous inequality over unit vectors, we deduce
$$
\displaystyle\inf_{|v|_2^2=1}D^2 W^N[v,v]\ge\displaystyle\lambda \left(1+ \frac{1}{N}\sup{_{\sum_{i=1}^N|v_i|_2^2=1}}\sum_{i\ne j} |v_i|^2+|v_j|^2\right)\ge 3\lambda.
$$
\end{proof}

Notice that the previous computations are reminiscent of estimates in \cite{Mal}.
The fact that the functional
$$
\mathcal{F}^N:\mathcal{P}(\Omega^N)\to \R\cup \{\infty\}
$$
is $\min(3\lambda,0)$-convex on geodesics and generalized geodesics of the 2-Wasserstein distance, follows from \cite[Propositions 9.3.2-9.3.5-9.3.9]{ambrosio2008gradient}, finishing the proof of Lemma \ref{lem: convexity}.

\section{Convergence of the metric and compactness in $\mathcal{P}(\mathcal{P}(\Omega))$}\label{sec:metricconvergence}
In this Section, we show first that the convergence introduced in Definition~\ref{def:convergence1} implies the convergence of all marginals of the $N$-particle distibutions $\mu^N$ as $N\to\infty$. Then, we will focus on the compactness of curves in $\P_{sym}(\Omega^N)$ towards elements in $\mathcal{P}(\mathcal{P}(\Omega))$ as $N\to\infty$.


\subsection{Equivalent characterizations of the metric}
\label{subsec:charact}
The point of this section is to give alternative characterizations to the convergence given in Definition~\ref{def:convergence1}. In this section, we show the following Lemma which can also be found in \cite{hauray2014kac}.
\begin{lemma}\label{lem:characterizations}\cite[Theorem 5.3]{hauray2014kac}
Given $X\in \mathcal{P}(\mathcal{P}(\Omega))$ and a sequence of symmetric probability measures $\{\mu^N\}_{N\in \N}$, then the following are equivalent
\begin{itemize}
\item[(i)] $\mu^N\to X$ in the sense of Definition~\ref{def:convergence1}.
\item[(ii)] For every $n\in \N$, the $n$-th marginal converges, that is to say,
\begin{equation*}
\lim_{N\to\infty}d_2^2(\mu^N_n, X^n)=0\qquad\mbox{for every $n\in\N$,}
\end{equation*}
where
\begin{equation*}
X^n=\int_{\P(\Omega)}\rho^{\otimes n}\;dX(\rho)\in \P_{sym}(\Omega^n).
\end{equation*}
\item[(iii)] The associated empirical distribution converges,
\begin{equation*}
\lim_{N\to \infty}\mathfrak{D}_2^2(\hat{\mu}^N,X)=0.
\end{equation*}
\end{itemize}
\end{lemma}

The proof of Lemma~\ref{lem:characterizations} can be found at the end of this section, after we introduce the necessary key observation obtained in \cite[Proposition 2.14]{hauray2014kac} that we reproduce here for the sake of completeness.

\begin{lemma}  \label{lem:HaurayMischler}\cite[Proposition 2.14]{hauray2014kac}
Using the previous notation, we have that 
\begin{equation*}
\frac{1}{N}d_2^2(\mu^N, \nu^N)=\mathfrak{D}_2^2(\hat{\mu}^N,\hat{\nu}^N).
\end{equation*}
In other words, the mapping induced by $T^N$ is a scaled isometry from $\mathcal{P}_{sym}(\Omega^N)$ to $\P(\P(\Omega))$.
\end{lemma}
\begin{proof}[Proof of Lemma~\ref{lem:HaurayMischler}]
\textit{Step 1.} We start by showing that
\begin{equation}\label{eq:firstineq}
\mathfrak{D}_2^2(\hat{\mu}^N,\hat{\nu}^N)\le \frac{1}{N}d_2^2(\mu^N, \nu^N).
\end{equation}

\medskip

We consider $\pi_0\in \pi[\mu^N, \nu^N]$, the optimal pairing. By taking the push forward we have $(T_N\times T_N)\#\pi_0=\Pi_0\in \Pi[\hat{\mu}^N,\hat{\nu}^N]$, and thus
$$
\mathfrak{D}_2^2(\hat{\mu}^N,\hat{\nu}^N) \le \int_{\P(\Omega)\times \P(\Omega)} d_2^2(\rho_1,\rho_2)\;d\Pi_0(\rho_1,\rho_2)\,.
$$
Computing the right-hand side, we get
\begin{equation*}
\int_{\P(\Omega)\times \P(\Omega)} d_2^2(\rho_1,\rho_2)\;d\Pi_0(\rho_1,\rho_2)=\int_{\Omega^N\times \Omega^N} d_2^2\left(\frac{1}{N}\sum_{i=1}^N \delta_{x_i},\frac{1}{N}\sum_{i=1}^N \delta_{y_i}\right)\;d\pi_0(x,y)\,.
\end{equation*}
Combining the previous equation, with the identity
\begin{equation*}
d_2^2\left(\frac{1}{N}\sum_{i=1}^N \delta_{x_i},\frac{1}{N}\sum_{i=1}^N \delta_{y_i}\right)=\min_{\sigma\in \Sigma}\frac{1}{N}\sum_{i=1}^N|x_i-y_{\sigma(i)}|^2\qquad\mbox{for any $x,$ $y\in\Omega^N$},
\end{equation*}
we have the desired inequality
\begin{equation*}
\begin{array}{rl}
\ds\mathfrak{D}_2^2(\hat{\mu}^N,\hat{\nu}^N)&\ds\le \frac{1}{N}\int_{\Omega^N\times \Omega^N} \min_{\sigma\in \Sigma}\sum_{i=1}^N(x_i-y_{\sigma(i)})^2\;d\pi_0(x,y)\\
&\ds=\frac{1}{N}\int_{\Omega^N\times \Omega^N} |x-y|^2\;d\pi_0(x,y)\\
&\ds=\frac{1}{N}d_2^2(\mu^N, \nu^N),
\end{array}
\end{equation*}
by using the symmetry of $\nu^N$, which shows~\eqref{eq:firstineq}.

\medskip

\textit{Step 2.} We now show the reversed inequality
\begin{equation}\label{eq:secondineq}
\mathfrak{D}_2^2(\hat{\mu}^N,\hat{\nu}^N)\ge \frac{1}{N}d_2^2(\mu^N, \nu^N).
\end{equation}

\medskip

We take $\Pi_1\in\Pi[\hat{\mu}^N,\hat{\nu}^N]\subset \P(\mathcal{K}^N,\mathcal{K}^N)$ the optimal pairing. Using the inverse of $T^N$, $T^{N(-1)}:\mathcal{K}_N\to \Omega^N$, we notice that
\begin{equation*}
\pi_1=(T^{N(-1)}\times T^{N(-1)})\# \Pi_1 \in \pi[\mu^N, \nu^N],
\end{equation*}
is an admissible pairing. Moreover, we have the identity
\begin{equation}\label{eq:forD}
\begin{array}{rl}
\ds\mathfrak{D}_2^2(\hat{\mu}^N,\hat{\nu}^N)&\ds=\int_{\mathcal{K}^N\times \mathcal{K}^N} d_2^2(\rho_1,\rho_2)\;d\Pi_1(\rho_1,\rho_2)\ds=\frac{1}{N}\int_{\Omega^N\times \Omega^N} \min_{\sigma\in \Sigma}\sum_{i=1}^N(x_i-y_{\sigma(i)})^2\;d\pi_1(x,y). 
\end{array}
\end{equation}
In what follows, we massage $\pi_1$ to show the desired inequality.

First, we symmetrize $\pi_1$. Given $\sigma\in \Sigma$ a permutation, we consider the mapping $U_\sigma:\Omega^N\to \Omega^N$, by $U_\sigma(x_1,...,x_N)=(x_{\sigma(1)},x_{\sigma(2)},...,x_{\sigma(N)})$. By symmetry, we have $U_\sigma\#\mu^N=\mu^N$ and $U_\sigma\#\nu^N=\nu^N$ for any $\sigma\in \Sigma$. Therefore, $(U_\sigma\times U_\sigma)_ \#\pi_1\in \pi[\mu^N, \nu^N]$. Therefore,
\begin{equation}\label{eq:defpi2}
\pi_2=\frac{1}{N!}\sum_{\sigma\in \Sigma}(U_\sigma\times U_\sigma)_ \#\pi_1\in \pi[\mu^N, \nu^N]
\end{equation}
is an admissible pairing. Moreover, the identity for $\mathcal{D}_2^2(\hat{\mu}^N,\hat{\nu}^N)$\eqref{eq:forD} also holds replacing $\pi_1$ with $\pi_2$.

Next, we consider the set and the family of measures given by
\begin{equation}\label{eq:cxy}
\begin{array}{c}
\ds\mathcal{C}_{x,y}=\{z\in \Omega^N\;: \; \exists \sigma\in \Sigma \;s.t.\; U_\sigma(y)=z\;\mbox{and}\; |x-z|^2=\min_{\sigma\in \Sigma}|x-U_\sigma(y)|^2\},\\
\ds \rho_{x,y}=\frac{1}{\#(\mathcal{C}_{x,y})}\sum_{z\in C_{x,y}}\delta_{(x,z)}\in \P(\Omega^N\times \Omega^N),
\end{array}
\end{equation}
where $\#(\mathcal{C}_{x,y})$ is the number of elements of $\mathcal{C}_{x,y}$. We notice that $\rho_{x,y}:\Omega^N\times \Omega^N \to \P(\Omega^N\times \Omega^N)$ is a Borel mapping. Hence, we can define
\begin{equation}\label{eq:defpi3}
\pi_3= \int_{\Omega^N\times \Omega^N} \rho_{x,y}\;d\pi_2(x,y),
\end{equation}
or alternatively, by duality, for $\psi\in C_b(\Omega^N\times \Omega^N)$
\begin{equation*}
\int_{\Omega^N\times \Omega^N} \psi(x,y)\; d\pi_3(x,y)=\int_{\Omega^N\times \Omega^N} \left(\frac{1}{ \#(\mathcal{C}_{x,y})}\sum_{z\in \mathcal{C}_{x,y}}\psi(x,z)\right)\; d\pi_2(x,y).
\end{equation*}
We now show that $\pi_3\in \pi[\mu^N, \nu^N]$ is an admissible transference plan. Taking $\vphi\in C_b(\Omega^N)$, we have
\begin{equation*}
\begin{array}{rl}
\ds\int_{\Omega^N\times \Omega^N} \vphi(x)\; d\pi_3(x,y)
&\ds=\int_{\Omega^N\times \Omega^N} \left(\frac{1}{ \#(\mathcal{C}_{x,y})}\sum_{z\in \mathcal{C}_{x,y}}\vphi(x)\right)\; d\pi_2(x,y)\\
&\ds=\int_{\Omega^N\times \Omega^N} \vphi(x)\; d\pi_2(x,y)\\
&\ds=\int_{\Omega^N}\vphi(x)\; d\mu^N(x),
\end{array}
\end{equation*}
which shows the first marginal. For the second marginal, we use the definition of $\pi_2$ \eqref{eq:defpi2} to obtain.
\begin{equation}\label{2marginal1}
\begin{array}{rl}
\ds\int_{\Omega^N\times \Omega^N} \vphi(y)\; d\pi_3(x,y)
&\ds=\int_{\Omega^N\times \Omega^N} \left(\frac{1}{ \#(\mathcal{C}_{x,y})}\sum_{z\in \mathcal{C}_{x,y}}\vphi(z)\right)\; d\pi_2(x,y)\\
&\ds=\int_{\Omega^N\times \Omega^N}\frac{1}{N!} \sum_{\sigma \in \Sigma}\left(\frac{1}{ \#(\mathcal{C}_{U_\sigma(x),U_\sigma(y)})}\sum_{z\in \mathcal{C}_{U_\sigma(x),U_\sigma(y)}}\vphi(z)\right)\; d\pi_1(x,y).
\end{array}
\end{equation}
From the definition of $\mathcal{C}_{x,y}$ we observe that given $\sigma\in \Sigma$ we have the following property $U_\sigma(z)\in \mathcal{C}_{U_\sigma(x),U_\sigma(y)}$, if and only if $z\in \mathcal{C}_{x,y}$. Hence,
\begin{equation}\label{2marginal2}
\begin{array}{c}
\ds\int_{\Omega^N\times \Omega^N}\frac{1}{N!} \sum_{\sigma \in \Sigma}\left(\frac{1}{ \#(\mathcal{C}_{U_\sigma(x),U_\sigma(y)})}\sum_{z\in \mathcal{C}_{U_\sigma(x),U_\sigma(y)}}\vphi(z)\right)\; d\pi_1(x,y)\\
=\\
\ds\int_{\Omega^N\times \Omega^N}\frac{1}{N!} \sum_{\sigma \in \Sigma}\left(\frac{1}{ \#(\mathcal{C}_{x,y})}\sum_{z\in \mathcal{C}_{x,y}}\vphi(U_{\sigma^{-1}}(z))\right)\; d\pi_1(x,y).
\end{array}
\end{equation}
Defining the symmetrization $\tilde{\vphi}(z)=\frac{1}{N!}\sum_{\sigma \in \Sigma}\vphi(U_{\sigma^{-1}}(z))$, we obtain the identities
\begin{equation}\label{2marginal3}
\frac{1}{N!} \sum_{\sigma \in \Sigma}\left(\frac{1}{ \#(\mathcal{C}_{x,y})}\sum_{z\in \mathcal{C}_{x,y}}\vphi(U_{\sigma^{-1}}(z))\right)=\frac{1}{ \#(\mathcal{C}_{x,y})}\sum_{z\in \mathcal{C}_{x,y}}\tilde{\vphi}(z)=\tilde{\vphi}(y)\,,
\end{equation}
where we have used that $\tilde{\vphi}(z)=\tilde{\vphi}(y)$ for all $z\in \mathcal{C}_{x,y}$ since $U_{\sigma}(y)=z$ by the definition of $ \mathcal{C}_{x,y}$, or in other words, the symmetry of $\tilde\vphi$ under permutations.

Putting \eqref{2marginal1}, \eqref{2marginal2} and \eqref{2marginal3} together and using that the second marginal of $\pi_1$ is $\nu^N$, we obtain 
\begin{equation*}
\int_{\Omega^N\times \Omega^N} \vphi(y)\; d\pi_3(x,y)=\int_{\Omega^N\times \Omega^N} \tilde{\vphi}(y)\; d\pi_1(x,y)=\int_{\Omega^N} \tilde{\vphi}(y)\; d\nu^N(y)=\int_{\Omega^N} \vphi(y)\; d\nu^N(y),
\end{equation*}
where the last identity follows from the symmetry of $\nu^N$. Hence, $\pi_3\in \pi[\mu^N, \nu^N]$ is an admissible pairing.

Using that $\pi_3\in \pi[\mu^N, \nu^N]$, its definition \eqref{eq:defpi3}, the definition of $ \mathcal{C}_{x,y}$ in \eqref{eq:cxy} and \eqref{eq:forD}, we have
\begin{equation*}
\begin{array}{rl}
	\ds\frac{1}{N}d_2^2(\mu^N,\nu^N)
    &\ds\le \frac{1}{N}\int_{\Omega^N\times\Omega^N} |x-y|^2\;d\pi_3(x,y)\\
    &\ds=\frac{1}{N}\int_{\Omega^N\times \Omega^N} \min_{\sigma\in \Sigma}\sum_{i=1}^N|x_i-y_{\sigma(i)}|^2\;d\pi_1(x,y)\\
    &\ds=\mathfrak{D}_2^2(\hat{\mu}^N,\hat{\nu}^N),
\end{array}
\end{equation*}
which shows the desired \eqref{eq:secondineq} and concludes the proof.
\end{proof}
To prove Lemma~\ref{lem:characterizations}, we need the following natural observation.
\begin{lemma}\label{lem:convergence}
Given $X\in\mathcal{P}_2(\mathcal{P}_2(\Omega))$, then
$$
\lim_{N\to \infty} \mathfrak{D}_2^2(\widehat{X^N},X)=0,
$$
where $\widehat{X^N}=T^N\#X^N$.
\end{lemma}
\begin{proof}
By the separability of the metric space $(\P(\Omega),d_2)$ we have compactness of measures. Therefore, for every sequence $N_i$ there exists a further subsequence (not relabeled) and a positive measure $Y\in\P(\P(\Omega))$ such that 
$$
\widehat{X^N}\rightharpoonup Y.
$$ 
Using Lemma~\ref{lem:characterizations}, we characterize $Y=X$ by showing the equality for the marginals. Given a smooth function $\vphi:\Omega^n\to \R$, we consider the action of the monomial $M_{n,\phi}$ on the sequence to obtain
$$
\begin{array}{rcl}
\ds\lim_{N\to\infty}\int_{\Omega^n} \vphi \;d (\widehat{X^N})_n&=&\ds \lim_{N\to\infty}\int_{\mathcal{P}(\Omega)} \langle\vphi,\rho^{\otimes n}\rangle_{C(\Omega^n),\P(\Omega^n)}\;d \widehat{X^N}(\rho),\\
\ds\int_{\Omega^n} \vphi\; d X^n &=&\ds\int_{\mathcal{P}(\Omega)} \langle\vphi,\rho^{\otimes n}\rangle_{C(\Omega^n),\P(\Omega^n)}\; d Y(\rho),
\end{array}
$$
where we have used the Diaconis-Freedman Lemma~\ref{lem:DiaconisFreedman} for the equality in the left hand side. So we can conclude that the full sequence $\widehat{X^N}\rightharpoonup X$. 

In the case we are working with a compact set $\Omega$, this is equivalent to showing that
$$
\lim_{N\to\infty}\mathfrak{D}_2^2(\widehat{X^N},X)=0.
$$
In the case $\Omega$ is unbounded, we also need to show that the second moment of the sequence converges. This follows from the following computation: for every $N\in \N$,
$$
\int_{\P(\Omega)} d_2^2(\rho,\delta_0)\; d\widehat{X^N}(\rho)=\int_{\Omega^N} \frac{|x|^2}{N}\; dX^N(x)=\int_{\Omega^N} |x_1|^2\; dX^1(x_1)=\int_{\P(\Omega)} d_2^2(\rho,\delta_0)\; dX(\rho).
$$
\end{proof}

\begin{proof}[Proof of Lemma~\ref{lem:characterizations}]
\textit{Step 1.} We show that (i) implies (ii).

\medskip

We take $\Pi\in \P(\Omega^N\times\Omega^N)$ the optimal pairing between $\mu^N$ and $X^N$. Fixing $n<N$ and denoting $\left\lfloor\frac{N}{n}\right\rfloor$ the integer part of $N/n$, we have by symmetry
\begin{equation*}
\int_{\Omega^N\times \Omega^N}                     |x-y|^2\;d\Pi=\left\lfloor\frac{N}{n}\right\rfloor\int_{\Omega^n\times \Omega^n} |\tilde{x}-\tilde{y}|^2\;d\Pi_n+\left(N-\left\lfloor\frac{N}{n}\right\rfloor\right)\int_{\Omega\times \Omega} |x_1-y_1|^2\;d\Pi_1,
\end{equation*}
where $\Pi_n\in \P(\Omega^n\times\Omega^n)$ and $\Pi_1\in\P(\Omega,\Omega)$ are the projections onto $n+n$ and $1+1$ variables, respectively. Using that $\Pi_n$ is an admissible pairing between $\mu_n^N$ and $\nu_n^N$, we obtain
\begin{equation*}
\frac{1}{N}d_2^2(\mu^N,X^N)\ge \frac{1}{N}\left\lfloor\frac{N}{n}\right\rfloor d_2^2(\mu_n^N,X^N_n)+\frac{N-\left\lfloor\frac{N}{n}\right\rfloor}{N}d_2^2(\mu^N_1,X^1).
\end{equation*}
Noticing that $X^N_n=X^n$, taking limits and using (i), we obtain that for every $n\in\N$, 
\begin{equation*}
0=\lim_{N\to\infty}\frac{1}{N}d_2^2(\mu^N,X^N)\ge \frac{1}{n}\lim_{N\to\infty}d_2^2(\mu^N_n, X^n)\ge0,
\end{equation*}
which implies (ii).

\medskip

\textit{Step 2.} We show that (ii) implies (iii).

\medskip

By the separability of the metric space $(\P_2(\Omega),d_2)$ we have the compactness of measures with finite mass. Therefore, for every subsequence $N_i\to\infty$, there exists a further subsequence (which we do not relabel) and $Y\in \mathcal{P}_2(\P_2(\Omega))$ such that
\begin{equation*}
\hat{\mu}^{N_i}\rightharpoonup Y.
\end{equation*}
We show that, independently of the subsequence, $Y=X$. We notice by the Diaconis-Freedman Lemma~\ref{lem:DiaconisFreedman} that for any $n\le N$,
\begin{equation*}
\|\mu^{N_i}_n-(\hat{\mu}^{N_i})^{n}\|_{TV}\le 2\frac{n(n-1)}{N}.
\end{equation*}
In particular, this implies that for every $n\in \N$,
$$
(\hat{\mu}^{N_i})^{n}\rightharpoonup X^n,
$$
and it follows that $Y=X$ by Lemma~\ref{lem:moments}.

To show the convergence of the metric $\mathfrak{D}_2$, we need to show that the second moment is also converging. To show this, we first notice that
$$
\int_{\P(\Omega)}d^2_2(\rho,\delta_0)\;d\hat{\mu}^{N}(\rho)=\int_{\Omega^N}\frac{|x|^2}{N}\;d\mu^N(x)=\int_{\Omega^N}|x_1|^2\;d\mu^N_1(x).
$$
Using the hypothesis (ii), we have the desired convergence
$$
\lim_{N\to\infty}\int_{\Omega^N}|x_1|^2\;d\mu^N_1(x)=\int_{\P(\Omega)}\left(\int_{\Omega}|x_1|^2\;d\rho(x)\right)\; dX(\rho)=\int_{\P(\Omega)}d_2^2(\rho,\delta_0)\; dX(\rho).
$$

\medskip

\textit{Step 3.} We show that (iii) implies (i).

\medskip

By Lemma~\ref{lem:HaurayMischler} and the triangle inequality we have
$$
\frac{1}{N}d_2^2(\mu^N, X^N)=\mathfrak{D}_2^2(\hat{\mu}^N,\widehat{X^N})\le2 \mathfrak{D}_2^2(\hat{\mu}^N,X)+2\mathfrak{D}_2^2(\widehat{X^N},X).
$$
The result follows by taking limits applying Lemma~\ref{lem:convergence} and using the hypothesis (iii).
\end{proof}
%
%
\subsection{Compactness of curves in $\P_{sym}(\Omega^N)$}

We show now a compactness result that will be useful for passing to the limit of solutions of the gradient flow.

\begin{lemma}[Compactness of $H^1$ curves]\label{lem:compactness1}
We fix $T>0$. Let $\{\mu^N(\cdot)\}_{N\in \N}$ be a family of curves such that for every $\mu_N:[0,T]\to \P_{sym}(\Omega^N)$, 
\begin{equation}\label{assumption H^1}
\sup_{N\in \N}\frac{1}{N}\int_0^T |\dot{\mu}^N|^2<\infty,
\end{equation}
holds, being $|\dot{\mu}^N|$ the metric derivative with respect to $d_2$ of $\mathcal{P}(\Omega^N)$.
Then, for every subsequence $N_i$, there exists a further subsequence $N_{i_j}$ and a curve $X:[0,T]\to \P(\P(\Omega))$ such that $\mu^{N_{i_j}}(\cdot)\to X(\cdot)$ uniformly in time in the sense of Definition~\ref{def:convergence1}. More precisely,
\begin{equation*}
\lim_{j\to \infty}\sup_{t\in[0,T]}\frac{1}{N_{i_j}}d_2^2(\mu^{N_{i_j}}(t), X^{ N_{i_j}}(t))=0,
\end{equation*}
and
\begin{equation*}
\liminf_{j\to \infty}\frac{1}{N_{i_j}}\int_0^T |\dot{\mu}^{N_{i_j}}(s)|^2\;ds\ge \int_0^T |\dot{X}(s)|^2\;ds,
\end{equation*}
where the metric derivative on the right hand side is with respect to $\mathfrak{D}_2^2$, the 2-Wasserstein distance on the probability measures of the metric space $(\mathcal{P}(\Omega),d_2)$.
\end{lemma}

\begin{proof}[Proof of Lemma~\ref{lem:compactness1}]
By Lemma~\ref{lem:HaurayMischler} and our assumption \eqref{assumption H^1}, the family $\{\hat{\mu}^{N}\}_{N\in\N}\subset \mathcal{P}(\mathcal{P}(\Omega))$ is uniformly bounded in $C^{1/2}([0,T];\mathcal{P}(\mathcal{P}(\Omega)))$ with respect to the metric $\mathfrak{D}_2$. By Arzela-Ascoli the family $\{\hat{\mu}^{N}\}_{N\in\N}$ is relatively compact. Hence, the existence of a curve $X:[0,T]\to \P(\P(\Omega))$ with the convergence up to a subsequence follows from the previous characterizations~Lemma~\ref{lem:characterizations}. For notational convenience, we forgo the subsequence notation.

Using that $\P(\P(\Omega))$ is a complete metric space, we can characterize the $H^1$ norm by the following supremum,
\begin{equation}\label{eq:H11}
\frac{1}{N}\int_0^T |\dot{\mu}^N|^2=\int_0^T |\dot{\hat{\mu}}^N|^2=\sup_{0<h<T}\int_0^{T-h} \frac{\mathfrak{D}_2^2(\hat{\mu}^N(s+h),\hat{\mu}^N(s))}{h^2}\;ds.
\end{equation}
Using the uniform convergence and Fatou's Lemma, we obtain that for any $0<h<T$,
\begin{equation}\label{eq:H12}
\liminf_{N\to\infty}\int_0^{T-h} \frac{\mathfrak{D}_2^2(\hat{\mu}^N(s+h),\hat{\mu}^N(s))}{h^2} \;ds\ge \int_0^{T-h} \frac{\mathfrak{D}_2^2(X(s+h),X(s))}{h^2}\;ds.
\end{equation}
Putting \eqref{eq:H11}, \eqref{eq:H12} together and taking the supremum over $h$ we obtain 
\begin{equation*}
\liminf_{N\to\infty}\frac{1}{N}\int_0^T|\dot{\mu}^N|^2\ge \sup_{0<h<T}\int_0^{T-h} \frac{\mathfrak{D}_2^2(X(s+h),X(s))}{h^2}\;ds=\int_0^T |\dot{X}|^2,
\end{equation*}
which is the desired lower-semicontinuity.
\end{proof}

Finally, we reinterpret, using the characterization of the metric that was presented in Section~\ref{subsec:charact}, the convergence of sequences in $\mathcal{P}_{sym}(\Omega^N)$ in terms of the 2-Wasserstein distance in $\mathcal{P}(\mathcal{P}(\Omega))$.

\begin{lemma}\label{lem:limitmetric}
Given two sequences of symmetric probability measures $\{\mu^N\}_{N\in\N}$ and $\{\nu^N\}_{N\in\N}$ such that $\mu^N,$ $\nu^N\in \mathcal{P}_{sym}(\Omega^N)$, if $\mu^N\to X$ and $\nu^N\to Y$ in the sense of Definition~\ref{def:convergence1}, then we have
\begin{equation*}
    \lim_{N\to\infty} \frac{1}{N} d_2^2(\mu^N,\nu^N)=\mathfrak{D}_2^2(X,Y),
\end{equation*}
where $\mathfrak{D}_2$ is the 2-Wasserstein distance on the probability measures of the metric space $(\mathcal{P}(\Omega),d_2)$.
In particular, we have
\begin{equation*}
    \lim_{N\to\infty} \frac{1}{N} d_2^2(\mu^N,Y^N)=\mathfrak{D}_2^2(X,Y),
\end{equation*}
where
\begin{equation*}
    Y^N=\int_{\mathcal{P}(\Omega)}\rho^{\otimes N}\;dY(\rho)\in \mathcal{P}_{Sym}(\Omega^N).
\end{equation*}
\end{lemma}

\begin{proof}
This follows from part $(iii)$ of Lemma~\ref{lem:characterizations}, together with the isometry property from Lemma~\ref{lem:HaurayMischler}.
\end{proof}


\section{$\Gamma$-Convergence of the Free Energy Functional}\label{sec:gammaconvergence}
In this section we prove $\Gamma$-convergence of the free energy functional for the $N-$particle system, in the spirit of the proof Messer-Spohn~\cite{messer1982statistical}, see also~\cite{Kiessling1993}. We follow the more recent proof by Rougerie~\cite[Chapter 2]{rougerie2015finetti}.

Let us first define the auxiliary functional $\mathcal{F}^\infty:\mathcal{P}(\mathcal{P}(\Omega))\to \R$ given by
$$
\mathcal{F}^\infty[X]=\int_{\mathcal{P}(\Omega)} \mathcal{F}^{MF}[\rho]\;dX(\rho),
$$
with $\mathcal{F}^{MF}:\mathcal{P}(\Omega)\to \R$ given by
$$
\mathcal{F}^{MF}[\rho]=\int_{\Omega} V(x)\;d\rho(x)+\frac{1}{2}\int_{\Omega}\int_{\Omega} H(x-y)\;d\rho(x)d\rho(y)+\int_{\Omega}\log(\rho(x))\;\rho(x).
$$

The objective of this section is to show the following $\Gamma$-convergence result.

\begin{lemma}\label{lem:gconv}
Given a sequence of symmetric probability measures $\{\mu^N\}_{N\in\N}$ such that $\mu^N\in \mathcal{P}_{sym}(\Omega^N)$, assume that there exists $X\in\mathcal{P}(\mathcal{P}(\Omega))$, such that $\mu^N\to X$ in the sense of Definition~\ref{def:convergence1}. Then we have:
    $$    \liminf_{N\to\infty}\frac{\mathcal{F}^N}{N}[\mu^N]\ge \mathcal{F}^\infty(X).$$
    Moreover, given $Y\in \mathcal{P}(\mathcal{P}(\Omega))$, we have
    $$
        \lim_{N\to\infty}\frac{\mathcal{F}^N}{N}[Y^N]=\mathcal{F}^\infty(Y),
    $$
    where 
    $$
        Y^N=\int_{\mathcal{P}(\Omega)}\rho^{\otimes N}\;dY(\rho)\in \mathcal{P}_{Sym}(\Omega^N).
    $$
\end{lemma}

We split the proof of this result into two parts, the existence of the recovery sequence and the lower-semicontinuity of the sequence of functionals.

\begin{proposition}[recovery sequence]
Given $Y\in \P(\P(\Omega))$, let
\begin{equation*}
\nu^N= \int_{\P(\Omega)} \rho^{\otimes N}\;dY(\rho).
\end{equation*}
Then,
\begin{equation*}
\lim_{N\to \infty} \frac{\mathcal{F}^N[\nu^N]}{N}= \mathcal{F}^\infty[Y].
\end{equation*}
\end{proposition}
\begin{proof}
By convexity of the function $x \log(x)$ and Jensen's inequality, we have
\begin{align}\label{aux1}
\frac1N \int_{\Omega^N}  \nu^N \log\nu^N\;dx & \le \frac1N \int_{\P(\Omega)} \left(\int_{\Omega^N} \rho^{\otimes N} \log(\rho^{\otimes N})\;dx \right)dY(\rho)\nonumber\\&= \int_{\P(\Omega)} \int_{\Omega} \rho \log(\rho)\;dx \,dY(\rho).
\end{align}
Moreover, we have
\begin{equation}\label{aux3}
 \frac{1}{N}\int_{\Omega^N} W^N d\nu^N= \int_{\P(\Omega)}\!\left(\int_{\Omega} V(x)\; d\rho(x)+\frac{N-1}{2N}\int_{\Omega\times \Omega}\!\!\!\!\! H(x-y)\;d\rho(x)d\rho(y)\right)\,dY(\rho). 
\end{equation}
Therefore, collecting terms in \eqref{aux1} and \eqref{aux3} and taking the limit, we obtain that
\begin{equation*}
\limsup_{N\to \infty} \frac{\mathcal{F}^N[\nu^N]}{N}\le \mathcal{F}^\infty[Y].
\end{equation*}
The fact that the limit converges follows from the lower semicontinuity property,~Proposition~\ref{prop:lsc}.
\end{proof}

\begin{proposition}[lower semicontinuity]\label{prop:lsc}
Given a sequence $\{\mu^N\}_{N=1}^\infty$ such that $\mu^N\in \P_{sym}(\Omega^N)$ and $X\in \mathcal{P}(\mathcal{P}(\Omega^N))$ such that $\mu^N\to X$ in the sense of Lemma~\ref{lem:HaurayMischler}, then
\begin{equation*}
\liminf_{N\to \infty}\frac{\mathcal{F}^N[\mu^N]}{N}\ge \mathcal{F}^\infty[X]= \int_{\mathcal{P}(\Omega)} \mathcal{F}^{MF}[\rho]\;dX(\rho),
\end{equation*}
where
\begin{equation*}
\mathcal{F}^{MF}[\rho]=\int_{\Omega} V(x)\;\rho(x)\;dx+\frac{1}{2}\int_{\Omega}\int_{\Omega} H(x-y)\;\rho(x) \rho(y)\;dx\;dy+\int_{\Omega}\rho(x)\log(\rho(x))\;dx.
\end{equation*}
\end{proposition}
\begin{proof}
Without loss of generality, up to subsequence which we do not relabel, we can assume that 
\begin{equation*}
\liminf_{N\to \infty}\frac{\mathcal{F}^N[\mu^N]}{N}=\lim_{N\to \infty}\frac{\mathcal{F}^N[\mu^N]}{N}\quad\mbox{and}\quad \sup_{N}\frac{\mathcal{F}^N[\mu^N]}{N}<\infty.
\end{equation*}
In particular, by lower semicontinuity with respect to weak convergence we have
\begin{align*}
\displaystyle\liminf_{N\to \infty}\frac{1}{N}\int_{\Omega^N}W^N\;d\mu^N(x)
&=\liminf_{N\to \infty}\frac{1}{2}\int_{\Omega\times\Omega} H(x-y)+V(x)+V(y)\; d\mu^N_2(x,y)\\
&\ge \frac{1}{2}\int_{\Omega\times\Omega} H(x-y)+V(x)+V(y)\; d\mu_2(x,y)\\
&=\int_{\mathcal{P}(\Omega)}\!\! \left(\frac{1}{2}\int_{\Omega\times\Omega} \!\!\! H(x-y)\; d\rho(x)d\rho(y)+\int_{\Omega}V(x)\;d\rho(x)\right)dX(\rho),
\end{align*}
which shows the desired inequality for the interaction and confinement term.

For the entropy term, we need to use the subadditivity property of the entropy~\cite{Lieb1975}. Let us consider the marginal $\mu^N_n$ as in \eqref{def:marginal}, that is integrating in the last $N-n$ variables, and we write
\begin{equation*}
\int_{\Omega^N} \mu^N \log(\mu^N)=\int_{\Omega^N} \mu^N \log\left(\mu^N_n\frac{\mu^N}{\mu^N_n}\right)=\int_{\Omega^N} \mu^N \left(\log\left(\mu^N_n\right)+ \log\left(\frac{\mu^N}{\mu^N_n}\right)\right)\,.
\end{equation*}
Integrating out the last $N-n$ variables in the first term, we obtain
\begin{equation*}
\int_{\Omega^N} \mu^N \log\left(\mu^N_n\right)=\int_{\Omega^n} \mu^N_n\log\left(\mu^N_n\right).
\end{equation*}
For the second term, we decompose it and apply Jensen's with respect to the probability measure $\mu^N_n$ to infer 
\begin{equation*}
\int_{\Omega^{N-n}}\left(\int_{\Omega^n}\mu^N \log\left(\frac{\mu^N}{\mu^N_n}\right)  \right)\ge \int_{\Omega^{N-n}}\left(\int_{\Omega^n}\mu^N\right)\log\left(\int_{\Omega^n}\mu^N\right)=\int_{\Omega^{N-n}}\mu^N_{N-n}\log\left(\mu^N_{N-n}\right).
\end{equation*}
where the symmetry of $\mu^N$ was used. Iterating this procedure and taking again into account the symmetry of $\mu^N$, we obtain the inequality
\begin{equation*}
\int_{\Omega^N} \mu^N \log(\mu^N)\ge \left\lfloor\frac{N}{n}\right\rfloor\int_{\Omega^n} \mu^N_n\log\left(\mu^N_n\right)+\int_{\Omega^{N-\left\lfloor\frac{N}{n}\right\rfloor n}}\mu^N_{N-\left\lfloor\frac{N}{n}\right\rfloor n}\log\left(\mu^N_{N-\left\lfloor\frac{N}{n}\right\rfloor n}\right).
\end{equation*}
Using the same procedure, with the first marginal, we obtain the inequality
\begin{equation*}
\int_{\Omega^{N-\left\lfloor\frac{N}{n}\right\rfloor n}}\mu^N_{N-\left\lfloor\frac{N}{n}\right\rfloor n}\log\left(\mu^N_{N-\left\lfloor\frac{N}{n}\right\rfloor n}\right)\ge \left(N-\left\lfloor\frac{N}{n}\right\rfloor n\right)\int_\Omega \mu^N_1 \log(\mu^N_1).
\end{equation*}
By the convergence $\mu^N\to X$ we know that $\lim_{N\to\infty}d_2^2(\mu^N_1,X^1)$, which implies the uniform bound
$$
\sup_{N\in \N}\int_{\Omega}|x_1|^2\;d\mu^N_1(x_1)\le C.
$$
By Carleman's inequality we have the uniform lower bound 
$$
\inf_{N\in \N}\int_\Omega \mu^N_1 \log(\mu^N_1)\ge -C.
$$
Hence, dividing by $N$, taking limits, using the lower semicontinuity of the entropy and the convergence $\mu^N\to X$, we obtain that for any $n\in \N$
\begin{equation*}
\liminf_{N\to \infty}\frac{1}{N}\int_{\Omega^N} \mu^N \log(\mu^N)\ge\frac{1}{n}\int_{\Omega^n} X^n \log(X^n).
\end{equation*}
Finally, to finish the proof we need to show the following property
\begin{equation}\label{eq:limitentropy}
\sup_{n\in \N}\frac{1}{n}\int_{\Omega^n} X^n \log(X^n)\ge \int_{\mathcal{P}(\Omega)}\left(\int_{\Omega}\rho\log(\rho)\right)\;dX(\rho).
\end{equation}
This was originally proven by Robinson and Ruelle in \cite{RobinsonRuelle}. The more modern proof that we present here can be found in \cite{hauray2014kac}. We first show that
\begin{equation}\label{eq:supandlim}
\sup_{n\in \N}\frac{1}{n}\int_{\Omega^n} X^n \log(X^n)=\lim_{n\to \infty}\frac{1}{n}\int_{\Omega^n} X^n \log(X^n).
\end{equation}
Given $\e>0$, there exists $j$ such that
\begin{equation*}
\frac{1}{j}\int_{\Omega^j} X^j \log(X^j)\ge \sup_{n\in \N}\frac{1}{n}\int_{\Omega^n} X^n \log(X^n)- \e.
\end{equation*}
Using the subadditivity of the entropy in the same procedure as before, we obtain that
\begin{align*}
\liminf_{k\to \infty}\frac{1}{k}\int_{\Omega^k} X^k \log(X^k) \ge \frac{1}{j}\int_{\Omega^j} X^j \log(X^j)\ge \sup_{n\in \N}\frac{1}{n}\int_{\Omega^n} X^n \log(X^n)- \e.
\end{align*}
Eqn.~\eqref{eq:supandlim} follows then by taking the limit $\e\to0^+$.

Let us now define the functional $\mathcal{E}:\P(\P(\Omega))\to \R\cup \{+\infty\}$ as
$$
\mathcal{E}(X):=\lim_{n\to\infty}\frac{1}{n}\int_{\Omega^n}X^n\log(X^n)\;dx 
$$
We notice that $\mathcal{E}(X)$ is linear over finite sums. Given $X,$ $Y\in\P(\P(\Omega))$, we have
$$
\begin{array}{c}
\ds\frac{1}{n}\int_{\Omega^n}\left(\frac{X^n+Y^n}{2}\right)\log\left(\frac{X^n+Y^n}{2}\right)\;dx\\
=\\
\ds\frac{1}{2n}\left(\int_{\Omega^n}X^n\log(X^n+Y^n)\;dx+\int_{\Omega^n}Y^n\log(X^n+Y^n)\;dx\right)-\frac{\log(2)}{n}\int_{\Omega^n}\left(\frac{X^n+Y^n}{2}\right)\\
\ge\\
\ds\frac{1}{2n}\left(\int_{\Omega^n}X^n\log(X^n)\;dx+\int_{\Omega^n}Y^n\log(Y^n)\;dx\right)-\frac{\log(2)}{n},
\end{array}
$$
where we have only used the standard properties of the logarithm. Taking the limit $n\to \infty$, we recover the inequality
$$
\mathcal{E}\left(\frac{X+Y}{2}\right)\ge \frac{1}{2}\mathcal{E}(X)+\frac{1}{2}\mathcal{E}(Y).
$$
The reverse inequality follows directly from convexity: for every $n\in\N$,
$$
\frac{1}{n}\int_{\Omega^n}\left(\frac{X^n+Y^n}{2}\right)\log\left(\frac{X^n+Y^n}{2}\right)\;dx\le \frac{1}{2n}\left(\int_{\Omega^n}X^n\log(X^n)\;dx+\int_{\Omega^n}Y^n\log(Y^n)\;dx\right).
$$
This implies
\begin{equation}\label{eq:linearity}
\mathcal{E}\left(\frac{X+Y}{2}\right)= \frac{1}{2}\mathcal{E}(X)+\frac{1}{2}\mathcal{E}(Y).    
\end{equation}
We notice that this readily implies that if we take a discrete measure $X_k= \sum_{i=1}^k\alpha_i \delta_{\rho_i}$, then
\begin{equation}\label{eq:discrete}
\mathcal{E}(X_k)=\sum_{i=1}^k\alpha_i\int_{\Omega}\rho_i(x)\log(\rho_i(x))\;dx=\int_{\P(\Omega)}\left(\int_{\Omega}\rho(x)\log(\rho(x))\;dx\right)\;dX_k(\rho).    
\end{equation}

We notice that, by Lemma~\ref{lem:characterizations}, $\mathcal{E}$ is the supremum of lower semicontinuous functionals with respect to the metric $\mathfrak{D}_2$ on $\P(\P(\Omega))$, therefore it is also lower semicontinuous. To conclude the proof, we find a sequence of discrete measures $\{X_k\}\subset \mathcal{P}(\mathcal{P}(\Omega))$ weakly converging to $X$ such that
$$
\mathcal{E}(X)\ge \mathcal{E}(X_k) \qquad\mbox{for every $k\in\N$.}
$$
From Theorem~\ref{thm:basicproperties} we know that $(P(P(\Omega),\mathfrak{D}_2)$ is a separable metric space. Hence, for any $\e>0$ we can cover $P(P(\Omega))$ with a countable number of balls of radius $1/k$ denoted by $\{B_i\}_{i=1}^\infty$. We pick $M\in \N$, such that
$$
\int_{\mathcal{P}_2(\Omega)\setminus \bigcup_{i=1}^{M-1}B_i}d^2_2(\rho,\delta_0)\;dX(\rho)<\frac{1}{k^2}.
$$
We define
$$
\omega_i=B_i\setminus \bigcup_{j=1}^{i-1}B_j,\qquad Z_i=\frac{1}{X(\omega_i)} X\measurerestr \omega_i,\qquad \rho_i=\int_{w_i}\rho\;dX(\rho),\qquad X_k=\sum_{i=1}^MX(\omega_i)\delta_{\rho_i}.
$$
Therefore, we get
$$
\mathcal{E}(X)=\sum_{i=1}^MX(\omega_i)\mathcal{E}(Z_i)\ge\sum_{i=1}^MX(\omega_i)\int_{\Omega}\rho_i(x)\log(\rho_i(x))\;dx=\mathcal{E}(X_k), 
$$
where we have used \eqref{eq:linearity}, Jensen's inequality and \eqref{eq:discrete}. The proof of \eqref{eq:limitentropy} follows by taking the limit when $k\to\infty$, noticing that by construction $\mathfrak{D}_2(X,X_k)\le 2/k$.
\end{proof}


\section{EVI Uniqueness and Proof of Theorem~\ref{main result}}\label{sec:uniquemain}

In this section we present the proof of the main result of this paper, Theorem~\ref{main result}. Let us first point out that we can define a unique gradient flow for evolutions in $\mathcal{P}(\mathcal{P}(\Omega))$.

\begin{lemma}\label{lem:uniquegradpp}
There exists a unique curve $X:[0,T]\to \mathcal{P}(\mathcal{P}(\Omega))$ satisfying
\begin{equation}\label{eq:EDI X}
\frac{e^{\lambda(t-s)}}{2}\mathfrak{D}_2^2(X(t),Y)-\frac{1}{2}\mathfrak{D}_2^2(X(s),Y)\le \left(\int_s^{t} e^{\lambda (r-s)}\;dr\right)(\mathcal{F}^{\infty}[Y]-\mathcal{F}^\infty[X(t)]),
\end{equation}
for any $0<s<t<\infty$ and
\begin{equation}\label{eq:initial X}
    X(0)=X_0=\lim_{N\to\infty} \mu_0^N.
\end{equation}
Moreover, it is explicitly given by
$$
X(t)=(S_t)_{\#}X_0,
$$
where $S_t:\P(\Omega)\to \P(\Omega)$ is the semigroup that is generated by the associated Fokker-Planck (McKean-Vlasov) equation 
\begin{equation*}
\partial_t \rho+\nabla\cdot((\nabla V +\nabla H *\rho) \rho)=\Delta \rho.
\end{equation*}
\end{lemma}

\begin{proof}

We differentiate
\begin{equation*}
\frac{e^{\lambda(t-s)}}{2}\mathfrak{D}_2^2(X(t),Y)-\frac{1}{2}\mathfrak{D}_2^2(X(s),Y)\le \left(\int_s^{t} e^{\lambda (r-s)}\;dr\right)(\mathcal{F}^{\infty}[Y]-\mathcal{F}^\infty[X(t)]).
\end{equation*}
to obtain the classical Evolutionary Variational Inequality which characterizes the gradient flows in metric spaces \cite{ambrosio2008gradient}. Uniqueness follows from using the doubling variables trick of Crandall-Liggett \cite[Chapter 4]{ambrosio2008gradient}.

We consider the Fokker-Planck semigroup $S_t:\P(\Omega)\to \P(\Omega)$ 
induced by the equation
\begin{equation*}
\begin{cases}
\partial_t \rho+\nabla\cdot((\nabla V +\nabla H *\rho) \rho)=\Delta \rho, \quad x \in \Omega,\\
\nabla \rho\cdot \overrightarrow{n}=0, \quad x \in \partial \Omega.
\end{cases}
\end{equation*}
Using $S_t$ and given $X_0\in \P(\P(\Omega))$, we can define the curve
\begin{equation*}
X_t=(S_t)_{\#}X_0 \,.
\end{equation*}
We claim that $X_t$ also satisfies the integral Evolutionary Variational Inequality. By using the $\lambda$ convexity of $\mathcal{F}^{MF}$ on the generalized geodesics, we have that for any $0<s<t<\infty$ and $\rho_1,$ $\rho_2\in \P(\Omega)\cap D\left(\mathcal{F}^{MF}\right)$ the inequality
\begin{equation}\label{eq:EVIrho}
\frac{e^{\lambda(t-s)}}{2}d_2^2(S^{t-s}\rho_1,\rho_2)-\frac{1}{2}d_2^2(\rho_1,\rho_2)\le \left(\int_s^{t} e^{\lambda (r-s)}\;dr\right)(\mathcal{F}^{MF}[\rho_2]-\mathcal{F}^{MF}[S^{t-s}\rho_1])
\end{equation}
holds, see \cite{ambrosio2008gradient}.

We consider $\Pi\in \P(\P(\Omega)\times\P(\Omega))$, the optimal pairing between $X_s$ and $Y$. We notice that $(S^{t-s}\times I)_\#\Pi$ is a pairing between $X_t$ and $Y$. Therefore, we have the inequality
\begin{equation*}
\begin{array}{rl}
\ds\frac{e^{\lambda(t-s)}}{2}\mathfrak{D}_2^2(X_t,Y)-&\ds\frac{1}{2}\mathfrak{D}_2^2(X_s,Y)\\
&\ds\le\int_{\P(\Omega)\times \P(\Omega)}\frac{e^{\lambda(t-s)}}{2}d_2^2(S^{t-s}\rho_1,\rho_2)-\frac{1}{2}d_2^2(\rho_1,\rho_2)\;d\Pi(\rho_1,\rho_2)\\
&\ds\le \int_{\P(\Omega)\times \P(\Omega)}\left(\int_s^{t} e^{\lambda (r-s)}\;dr\right)(\mathcal{F}^{MF}[\rho_2]-\mathcal{F}^{MF}[S^{t-s}\rho_1])\;d\Pi(\rho_1,\rho_2)\\
&\ds=\left(\int_s^{t} e^{\lambda (r-s)}\;dr\right)(\mathcal{F}^{\infty}[Y]-\mathcal{F}^\infty[X_t]),
\end{array}
\end{equation*}
where we have used \eqref{eq:EVIrho}, the fact that $(S^{t-s}\times I)_\#\Pi$ is a pairing between $X_t$ and $Y$ and the definition of $\mathcal{F}^\infty$.
\end{proof}

We are now ready to prove the main result of this paper.

\begin{proof}[Proof of Theorem~\ref{main result}]
We first use Lemma \ref{lem: convexity} to show the convexity of $\frac{1}{N}\mathcal{F}^N$ along generalized geodesics in $\P_{sym}(\Omega^N)$ for all $N\in\N$. Next, we use arguments from the theory of gradient flows. The following result can be found for instance in \cite[Theorem 4.0.4, Theorem 11.2.1]{ambrosio2008gradient}, \cite[Theorem 4.20]{daneri2010lecture} or \cite{santambrogio2017euclidean,savare2007gradient}.

\begin{theorem}\label{thm:Ngradientflow1}
Given $\mu^N_0\in \P_{sym}(\Omega^N)\cap Domain(\mathcal{F}^N)$, then there exists $\mu^N:[0,\infty)\to \P(\Omega^N)$ the unique gradient flow of $\mathcal{F}^N$, such that 
$$
\lim_{t\to 0}d_2(\mu^N,\mu^N_0)=0.
$$
Moreover, it satisfies
\begin{itemize}
    \item $\mu^N(t)\in \P_{sym}(\Omega^N)$ for any $t\ge 0$.
    \item The Energy Disipation Equality (EDE)
\begin{equation}\label{eq:EDEthm1}
\mathcal{F}^N[\mu^N(t)]+\frac{1}{2}\int_0^t |\dot{\mu}^N(s)|^2\;ds+\frac{1}{2}\int_0^t |\partial \mathcal{F}^N[\mu^N(s)]|^2\;ds=\mathcal{F}^N[\mu_0^N],
\end{equation}
for any $t\ge0$.
\item The integral Energy Variational Inequality (EVI)
\begin{equation}\label{eq:integralEVIthm1}
\frac{e^{3\lambda(t-s)}}{2}d_2^2(\mu^N(t),\nu^N)-\frac{1}{2}d_2^2(\mu^N(s),\nu^N)\le \left(\int_s^{t} e^{3\lambda (r-s)}\;dr\right)(\mathcal{F}^N[\nu^N]-\mathcal{F}^N[\mu^N(t)])
\end{equation}
holds for any $0<s<t<\infty$ and $\nu^N\in \mbox{Domain}(\mathcal{F}^N)$.
\end{itemize}
\end{theorem}

The next step in the proof is to make use of the EDE \eqref{eq:EDEthm1} to gain compactness of the curves $\{\mu^N\}_{N\in \N}$ by Lemma \ref{lem:compactness1}. Once we have a limiting evolution $X(t)$ in $\mathcal{P}(\mathcal{P}(\Omega))$, we need to pass to the limit in the EVI. We first notice the convergence of the metric given in Lemma \ref{lem:limitmetric} giving the convergence of the lefthand side of the EVI \eqref{eq:integralEVIthm1}. The convergence of the right-hand side of the EVI \eqref{eq:integralEVIthm1} is given by the $\Gamma$-convergence result proved in Lemma \ref{lem:gconv}. Therefore, by taking the limit in the EVI \eqref{eq:integralEVIthm1}, we have that the curve $X(\cdot)$ satisfies
\begin{equation*}
\frac{e^{\lambda(t-s)}}{2}\mathfrak{D}_2^2(X(t),Y)-\frac{1}{2}\mathfrak{D}_2^2(X(s),Y)\le \left(\int_s^{t} e^{\lambda (r-s)}\;dr\right)(\mathcal{F}^{\infty}[Y]-\mathcal{F}^\infty[X(t)]),
\end{equation*}
for any $0<s<t<\infty$ and
\begin{equation*}
    X(0)=X_0=\lim_{N\to\infty} \mu_0^N.
\end{equation*}
We finish the proof of our main result by using the uniqueness part of Lemma~\ref{lem:uniquegradpp}, identifying our limiting evolution as the gradient flow solution in $\mathcal{P}(\mathcal{P}(\Omega))$.
\end{proof}

 \bibliographystyle{abbrv}
 \bibliography{biblio}

\end{document}